\documentclass[11pt,leqno]{amsart}
\usepackage{amsmath,amssymb,amsthm,graphicx,enumerate,cite}
\usepackage[percent]{overpic}

\DeclareMathOperator*{\Argmax}{Arg\, max}
\DeclareMathOperator*{\ext}{ext}
\DeclareMathOperator*{\cl}{cl}
\DeclareMathOperator*{\Limsup}{Lim\,sup}
\DeclareMathOperator*{\lspan}{span}

\DeclareMathOperator*{\inte}{int}
\DeclareMathOperator*{\relint}{relint}
\DeclareMathOperator*{\conv}{conv}
\DeclareMathOperator*{\cone}{cone}

\DeclareMathOperator{\Proj}{\Pi}
\DeclareMathOperator*{\mT}{Tangent}
\DeclareMathOperator{\mD}{Dir}
\DeclareMathOperator*{\mN}{Normal}
\DeclareMathOperator*{\face}{face}

\newcommand{\E}{\mathbb{E}}
\newcommand{\N}{\mathbb{N}}
\newcommand{\R}{\mathbb{R}}

\newcommand{\cK}{\mathcal{K}}

\newcommand{\cT}{\mathcal{T}}

\newtheorem{proposition}{Proposition}
\newtheorem{theorem}{Theorem}

\theoremstyle{definition}

\newtheorem{example}{Example}

\newcommand{\iprod}[2]{\left\langle {#1}, {#2} \right\rangle}

\newcounter{claim_nb}[theorem]
\begin{document}

\title[Facially dual complete cones, lexicographic tangents]{Facially Dual Complete (Nice) cones\\ and \\lexicographic tangents}

\author{Vera Roshchina
\and Levent Tun\c{c}el}\thanks{Vera Roshchina: (v.roshchina@unsw.edu.au)  School of Mathematics and Statistics, University of New South Wales, Australia; much of the work on this paper was done while this author was also affiliated with RMIT University and Federation University Australia.
Research of this author was supported by the Australian Research Council
(Discovery Early Career Researcher Award DE150100240).\\
Levent Tun\c{c}el: (ltuncel@uwaterloo.ca)
Department of Combinatorics and Optimization, Faculty of
Mathematics, University of Waterloo, Waterloo, Ontario N2L 3G1,
Canada. Research of this author was supported in part by Discovery
Grants from NSERC and by U.S. Office of Naval Research under award numbers: N00014-12-1-0049 and
N00014-15-1-2171.}

\begin{abstract}
We study the boundary structure of closed convex cones, with a focus on
facially dual complete (nice) cones. These cones form a proper subset of facially exposed convex cones,
and they behave well in the context of duality theory for convex optimization. Using the well-known
and commonly used concept of tangent cones in nonlinear optimization,
we introduce some new notions for exposure of faces of convex sets.  Based on these new notions, we obtain
a necessary condition and a sufficient condition for a cone to be facially dual complete. In our sufficient
condition, we utilize a new notion called lexicographic tangent cones (these are a family of cones
obtained from a recursive application of the tangent cone concept).
Lexicographic tangent cones are related to Nesterov's lexicographic derivatives and
to the notion of subtransversality in the context of variational analysis.
\end{abstract}

\date{April 20, 2017, revised: December 28, 2018}



\maketitle

\section{Introduction}

Understanding the facial structure of convex cones as it relates to the dual cones
is fundamentally useful in convex optimization and analysis.  Let $K$ be a closed
convex cone in a finite dimensional Euclidean space $\E$. For a given scalar
product $\iprod{\cdot}{\cdot}$, the dual cone is
\[
K^*:= \left\{ s \in \E^*: \iprod{s}{x} \geq 0 \,\,\,\, \forall x \in K \right\},
\]
where $\E^*$ denotes the dual space.
Let $C\subseteq \E$ be a closed convex set. A closed convex subset  $F\subseteq C$ is called a {\em face}
of $C$ if for every $x\in F$ and every $y,z \in C$ such that $x\in (y,z)$, we have $y,z\in F$.
The fact that $F$ is a face of $C$ is denoted by $F\unlhd C$. Observe that the empty set and the set $C$ are both faces of $C$.
Just like other partial orders in this paper, if we write $F \lhd C$, then we mean $F$ is a face of $C$ but is not equal to $C$.
A nonempty face $F\lhd C$ is called {\em proper}. Note that if $K$ is a closed convex cone
and $F \unlhd K$, then $F$ is a closed convex cone.

We say that a face $F$ of a closed convex set $C$ is {\em exposed}\index{exposed face} if there exists a supporting hyperplane $H$ to the set $C$
such that $F = C\cap H$.  Many convex sets have unexposed faces, e.g., convex hull of a torus (see Fig.~\ref{fig:Tablet}).
\begin{figure}[ht]
	\centering
	\includegraphics[height = 90pt]{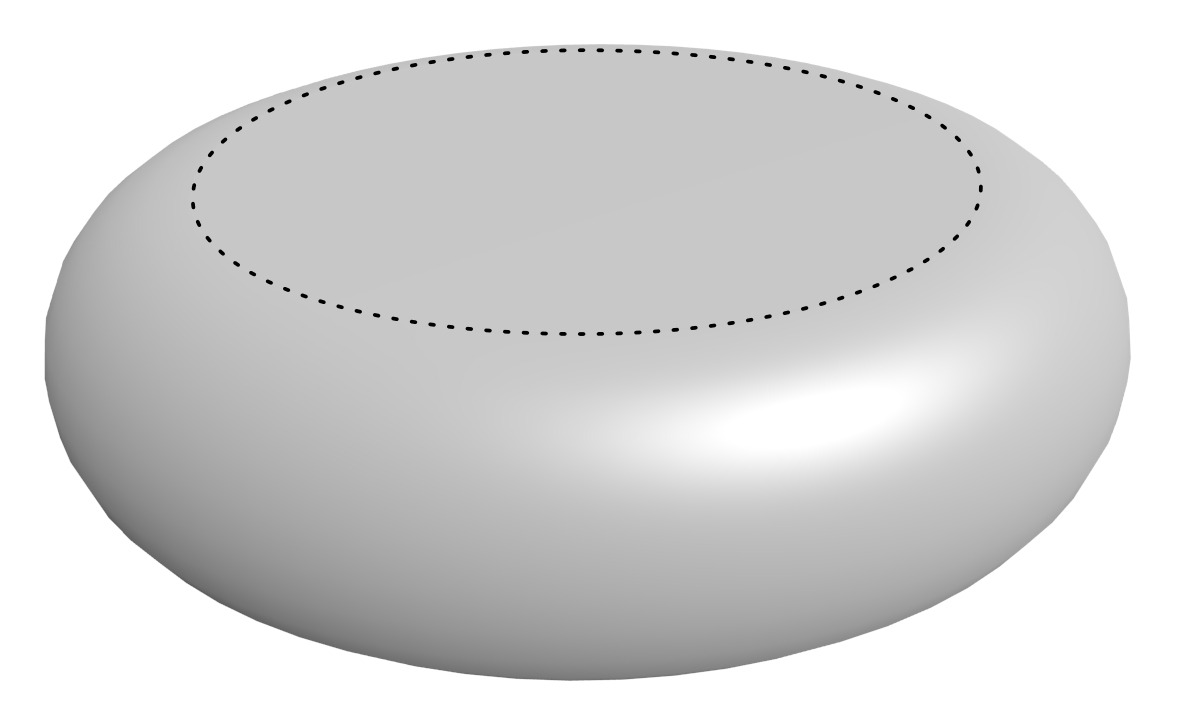}
	\caption{Convex hull of a torus is not facially exposed: the dashed line shows the set the extreme points which are not exposed (see \cite{RockConvAn}).}
		\label{fig:Tablet}
\end{figure}
Another example of a convex set with unexposed faces is the convex hull of a closed unit ball and a disjoint point 
(see for instance \cite{PatakiFexpNice} and Fig.~\ref{fig:nonexp2d} here).
\begin{figure}[ht]
	\centering
\begin{overpic}[scale=1
]{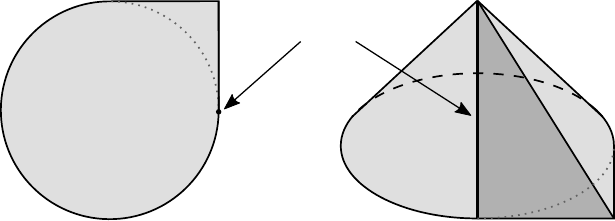}
 \put (42,31) {unexposed faces}
\end{overpic}
	\caption{An example of a two dimensional set and a three dimensional cone that have an unexposed face.}
	\label{fig:nonexp2d}
\end{figure}

 A closed convex set is \emph{facially exposed} if every proper face of $C$ is exposed.
\emph{Facial exposedness} is fundamental in understanding the boundary structure of convex
sets; it even has consequences in the theory of convex representations \cite{ChuaTuncel2008,GPT2013}.
\emph{Symmetric cones} and \emph{homogeneous cones} are facially exposed (see \cite{FK1994,TuncelXu2001,TruongTuncel2004}).
\emph{Hyperbolicity cones} are facially exposed too \cite{Renegar2006}, and they represent a powerful and interesting generalization
of symmetric cones and homogeneous cones for convex optimization
\cite{Guler1997,Renegar2006} and for many other research areas.

Now we turn to another property of faces.  We first motivate the concept and then define it rigorously.
Suppose that for a given family of convex optimization problems in conic form, we know that there is at least an
optimal solution that is contained in a face $F$ of $K$.  We may not have a direct access to the face $F$,
but perhaps we know the linear span of the face $F$: $\lspan(F)$.  Then, to compute an optimal solution, we may replace the
cone constraint $x \in K$, by $x \in \left(K\ \cap \lspan(F)\right)$.  Now, if we write down the dual problem, the dual cone constraint
(for the dual slack variable $s$) becomes (see Proposition~\ref{prop:NormalIntersection}):
\[
s \in \left(K \cap \lspan(F)\right)^* = \cl \left(K^* + F^{\perp}\right)
\]
where $F^{\perp} := \left\{ s \in \E^* : \,\, \iprod{s}{x} = 0 \,\,\,\, \forall x \in F \right\}$.
Indeed, if $\left(K^* + F^{\perp}\right)$ happens to be closed,
then we can remove the closure operation; otherwise, we would have to deal with this closure operation in some way.
Beginning with this observation, we have our first hints for the uses of the concept of
\emph{Facially Dual Complete} convex cones.
Closed convex cones $K$ with the property that
\[
\left(K^* + F^{\perp}\right) \textup{ is closed for every proper face } F \lhd K,
\]
are called \emph{Facially Dual Complete (FDC)}.  Pataki \cite{Pataki2007,PatakiFexpNice} called such cones \emph{nice}.
FDC property is one of the main concepts that we study in this paper.  Our interest in FDCness is motivated
by many factors:
\begin{itemize}
\item
FDC property is very important in duality theory.
Presence of facial dual completeness makes various facial reduction algorithms behave
well, e.g. see Borwein and Wolkowicz \cite{BW1981}, Waki and
Muramatsu \cite{WakiMuramatsu2013} and Pataki \cite{Pataki2013} (where it is shown explicitly how facial reduction can be specialised for the case of FDC cones).  Currently, the only exact
characterization of FDCness is via facial reduction (see Liu and Pataki~\cite{LiuPataki2015b}).
For some other recent work related to
facial reduction, see \cite{Krislock2010,WakiMuramatsu2010,Krislock2012,Vris,Pataki2013,Permenter2014,Permenter2015,LiuPataki2015a,DPW2015,Lourenco2015,PPFRA}.
\item
FDC property is also relevant in the fundamental subject of closedness of the image
of a convex set under a linear map. See Pataki \cite{Pataki2007}
and the references therein.
\item
FDC property comes up in the area of lifted convex representations (see \cite{GPT2013}) and
in representations of a family of convex cones as a slice of another family of convex cones (see \cite{ChuaTuncel2008}).
\item
FDC property seems to have a rather mysterious connection (see Pataki \cite{PatakiFexpNice})
to facial exposedness of the underlying cone which
is an intriguing and rather beautiful geometric property.  Moreover, better understanding
of FDC property contributes to our understanding of the boundary structure of convex sets.
\end{itemize}

Our paper is organized as follows. In Section~\ref{sec:char} we recall some notation and some of the known results related to the facial structure of convex cones, then we state and prove the necessary and sufficient conditions for facial dual completeness (Theorems~\ref{thm:NecFDC} and~\ref{thm:SuffFDC}). Throughout this process, we introduce some new notions for exposure of faces.  In Figure~\ref{fig:Tablet1} we summarize some of the relationships among various exposure properties.  Up to and including 3-dimensions, for convex cones, all of the four properties we listed in Fig.~\ref{fig:Tablet1} are precisely the same.  Starting in 4-dimensions, these four properties
identify different sets of convex cones.  We are able to illustrate these 4-dimensional convex cones, by taking 3-dimensional slices.

\begin{figure}[ht]
	\centering
\begin{overpic}[scale=1
]{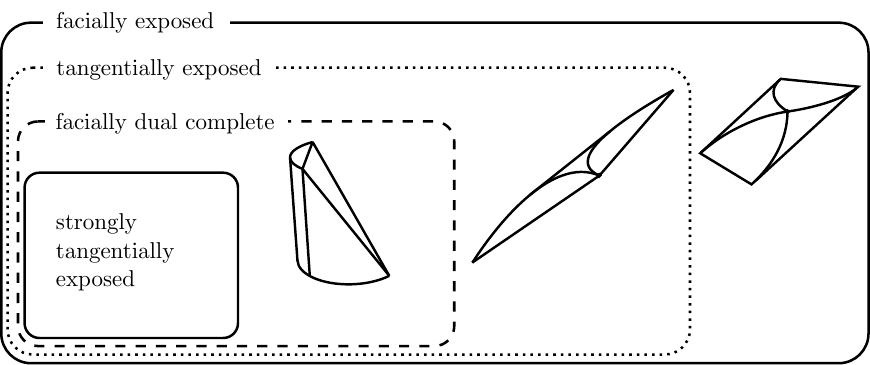}
 \put (35,5) {Example~\ref{eg:lost-tangent}}
 \put (60,5) {Example~\ref{eg:cubic}}
 \put (84,5) {Example~\ref{eq:one}}
\end{overpic}
	\caption{Relationships among various notions of facial exposure and FDCness. The graphics represent the examples discussed in this paper.}
	\label{fig:Tablet1}
\end{figure}

\bigskip
\section{Preliminaries}

Let $\E$ denote a finite dimensional Euclidean vector space, and let $\E^*$ be its dual.  Throughout this section by $K$  we denote a closed convex cone in $\E$.  We call $K$ \emph{regular} if $K$ is pointed (does not contain whole lines), closed, convex and has nonempty interior in $\E$.  If $K$ is a regular cone then so is its dual cone $K^*$.

Let $C\subseteq \E$ and $x\in C$. The \emph{cone of feasible directions of $C$ at $x$} is
\[
\mD(x;C):= \left\{d \in \E \, : \, \left(x +\epsilon d\right) \in C \textup{ for some } \epsilon >0 \right\}.
\]
The \emph{tangent cone} for $C$ at $x$ is
\[
\mT(x;C) := \cl \mD(x;C).
\]
Note that this definition can be restated in terms of the Painlev\'e--Kuratowski outer limit (see \cite{RockWets}),
$$
\mT(x;C)  = \Limsup_{t\to +\infty} t(C-x).
$$

The direction $s\in \E^*$ is said to be {\em normal} to a closed convex set $C$ at a point $x$ if
$$
\langle s, y-x\rangle \leq 0 \quad \forall\, y \in C.
$$
The set of all such directions is called the \emph{normal cone at $x$ to $C$}, denoted by $\mN(x;C)$.

In addition to the notion of dual cone, we also use the closely related concept of \emph{polar} of a set.
For a subset $C$ of $\E$, the \emph{polar} of $C$ is
\[
C^{\circ}:= \left\{ s \in \E^* : \iprod{s}{x} \leq 1 \,\,\,\, \forall x \in C \right\}.
\]
Note that for cones the notions of dual cone and polar are equivalent.  For example,
for every convex set $C$ and for every $x \in C$, we have
\[
\mN(x;C)= \left[\mT(x;C)\right]^{\circ} \textup{ and } \mT(x;C)= - \left[\mN(x;C)\right]^*.
\]
The following fact is used many times in this paper.
\begin{proposition}\label{prop:NormalIntersection} For every pair of closed convex cones $K_1$ and $K_2$ in $\E$, we have
\[
(K_1\cap K_2)^* = \cl \left(K^*_1+K_2^*\right).
\]
If the relative interiors of $K_1$ and $K_2$ have nonempty intersection, then
$K^*_1+K_2^*$ is a closed set and therefore the closure operation can be omitted.
\end{proposition}
\begin{proof}	
See Corollary 16.4.2 in Rockafellar \cite{RockConvAn} and Remark~5.3.1. in \cite{JBHU}.
\end{proof}

Our results can be established in a coordinate-free way by keeping the operations on
sets in the primal space and the dual space separate\footnote{
Let $F \subset \E$. Then we may consider the dual cone
of $F$ with respect to any Euclidean space $L$ such that $\lspan(F) \subseteq L \subseteq \E$.
We could denote by $F|_{L}^*$ the dual cone of $F$ in $\E^*/L^{\perp}$; i.e.,
\[
F|_{L}^* :=  \left\{ s \in \E^*/L^{\perp} : \,\, \iprod{s}{x} \geq 0 \,\,\,\, \forall x \in F \right\}.
\]
Next, we would define the projection map in the dual space.
For $C \subseteq \E^*$,
\[
\Pi_{\E^*/L^{\perp}}(C) := \left\{ [v] : v \in C \right\},
\]
where $[v]$ is the equivalence class of $v \in \E^*$ with respect to $L^{\perp}$.}.
However, for reducing the amount of
notation and for better readability, we pick a basis for $\E$, define an inner product on $\E$
from the scalar product above so that with this fixed inner-product $\E=\E^*=\R^n$.  From now on, 
$\iprod{\cdot}{\cdot}$ denotes an inner-product on $\R^n$.

Let $C$ be a closed convex set and let $S$ be a nonempty subset of $C$. We define the {\em minimal face}\index{minimal face} of $C$ containing $S$ as follows:
$$
\face(S;C) := \bigcap \{ F\,: \, F\unlhd C, S\subseteq F \}.
$$	
The following facts are elementary (and a few are well-known), we present all but one without proof. For $u \in \R^n$, we denote
\[
u^{\perp} := \left\{ x \in \R^n \,\, : \,\, \iprod{u}{x} = 0 \right\}.
\]
\begin{proposition}[Properties of faces]\label{prop:ff} Let $C$ be a closed convex set in $\R^n$. Then the following properties are true:
	\begin{itemize}
		\item[(i)] face of a face of $C$ is a face of $C$ (i.e., $G \unlhd F\unlhd C$ implies $G\unlhd C$);
		\item[(ii)] for every $x\in C$ and every $u\in \mN(x;C)$ with $F:= \face(\{x\},C)$, the set $\mT(x;F) \cap u^{\perp}$ is a face of $\mT(x;F)$;
		\item[(iii)]for every $S \subseteq C$, we have $\relint \left(\conv S\right) \cap \relint \left(\face(S;C)\right) \neq \emptyset$.
	\end{itemize}
\end{proposition}

\begin{proposition}\label{prop:SupSetTang} Let $K$ be a closed convex cone in $\R^n$.  Then, for every pair
$(u,x)$ with $u \in K^*$ and $x \in \left(K\cap u^{\perp}\right)$, with $F:= \face(\{x\},K)$, we have $u\in \left[\mT(x;F)\right]^*$.
\end{proposition}
\begin{proof}	
	Since $u$ defines a supporting hyperplane to $F$ at $x$, this supporting hyperplane is also a
supporting hyperplane for the tangent cone, and hence $u\in \left[\mT(x;F)\right]^*$.
\end{proof}

\begin{proposition}\label{prop:CharNice} A closed convex cone $K\subseteq \R^n$ is FDC if and only if for every face $F\lhd K$
	$$
	F^* \cap \lspan F = \Pi_{\lspan F}(K^*).
	$$
\end{proposition}

Here by $\Pi_{L}$ we denote the orthogonal projection onto a linear subspace $L\subseteq \R^n$, i.e. for each $x\in \R^n$ the projection $p = \Pi_L(x)$ is the unique point $p\in L$ such that
$$
\|p-x\|= \min_{y\in L}\|y-x\|.
$$
Above, we used the Euclidean norm induced by the inner product, hence, for $p=\Pi_L(x)$ we have, in particular, $(x-p) \in L^\perp$, a fact utilised heavily in the sequel.

\section{Facially Dual Complete Cones and Tangential Exposure}\label{sec:char}

We say that a closed convex  set $C$ in $\R^n$ has {\em tangential exposure} property if
\begin{equation}\label{eq:GeomCond}
	\mT(x;C)\cap \lspan (F-x) = \mT(x;F)\quad \forall F \lhd C,\; \forall x\in F.
\end{equation}
If $C$ is a convex cone then $\lspan (F-x) = \lspan F$ for every $x \in F$. So, in this special
case, we may write $\lspan F$ instead of $\lspan (F-x)$.

Tangential exposure is a stronger property than facial exposure. We discuss the relation between these two notions and provide illustrative examples later in this section.
Tangential exposure property can be equivalently characterised as \emph{subtransversality} of the set $C$ and the affine span of the face $F$ (see \cite{Ioffe2017}). We also note that
while this paper was being revised, a similar condition was used to derive error bounds for conic problems \cite{LourencoErrorBounds}.
Next, we prove Theorem~\ref{thm:NecFDC} which gives a necessary condition for the FDC property, establishing that every FDC cone is tangentially exposed.

\subsection{Proof of the necessary condition}\label{sec:NecProof}

\begin{theorem}
	\label{thm:NecFDC}
	If a closed convex cone $K \subseteq \R^n$ is facially dual complete, then for every $F \lhd K$ and every $x\in F$, we have
	\begin{equation}\label{eq:GeomCond}
		\mT(x;K)\cap \lspan F = \mT(x;F).
	\end{equation}
\end{theorem}
\begin{proof}
	Since $\mT(x;F)$ is a subset of both $\mT(x;K)$ and $\lspan F$, the inclusion
\[
\mT(x;K)\cap \lspan F \supseteq \mT(x;F)
\]
follows.
	For the reverse inclusion, for the sake of reaching a contradiction,
	assume the contrary: $K$ is facially dual complete, but there exist $F\triangleleft K$ and $x\in F$ such that \eqref{eq:GeomCond} does not hold.
	Then, there exists $g\in \mT(x;K)\cap \lspan F $ such that $g\notin \mT(x;F)$. Without loss of generality, we may assume $\|g\|=1$.
	Since $g\in \lspan F=:L$, applying the hyperplane separation theorem to $g$ and $\mT(x;F)$,
	in the space of $\lspan F$, we deduce that there exists $p\in \mN(x;F) \cap L$ such that $\langle p,g\rangle > 0$.
	
	Since $F$ is a cone, we have $\mN(x;F)\subseteq \mN(0;F)=-F^*$, hence, $p\in -F^*$.
	Since $K$ is facially dual complete, by Remark~1 in \cite{PatakiFexpNice} we have $F^* = K^*+F^\perp$; hence, there exist
	$y\in -K^*$ and $z\in F^\perp$ such that $y = p-z$. Since $g\in \lspan F$ and $z\in F^\perp$, we have
	$$
	\langle y, g\rangle = \langle p-z,g\rangle  = \langle p,g\rangle>0.
	$$
	Since $g\in \mT(x;K)$, there exists a sequence $\{s_k\}$, such that $s_k\in K$ and
	$$
	\lim_{k\to\infty} \frac{s_k-x}{\|s_k-x\|} = g.
	$$
	Therefore,
	$$
	\lim_{k\to\infty} \frac{\langle s_k-x,y\rangle }{\|s_k-x\|} = \langle g,y\rangle >0,
	$$
	and there exists $k$ large enough such that
	$$
	\langle s_k-x,y\rangle>0.
	$$
	Now observe that since $F$ is a cone, and $x\in F$, we also have $\frac{1}{2} x\in F$ and $\frac{3}{2}x\in F$, hence, by the definition of the tangent cone,
	$$
	-\frac{1}{2}x,\frac{1}{2}x \in \mT(x;F).
	$$
	Since $p\in \mN(x;F)$, this yields $\langle p,x\rangle = 0$. Then $\langle x,y\rangle = \langle x,p\rangle - \langle x,z\rangle = 0$, and we have
	$$
	0 <\langle s_k-x,y\rangle = \langle s_k,y\rangle.
	$$
	However, this is impossible, as $s_k\in K$, $y\in -K^*$, and hence $\langle s_k,y\rangle \leq 0$. Therefore, our assumption is not true,
	and by the arbitrariness of $F$ and $x$ we have shown that \eqref{eq:GeomCond} holds for all $F\lhd K$ and all $x\in F$.
\end{proof}

For the sake of completeness of our exposition, we prove that the tangential exposure yields facial exposure.

\begin{proposition}\label{prop:texpexp} Let $C\subseteq\R^n$ be a closed, convex, tangentially exposed  set. Then every proper face $F\lhd C$ is exposed.
\end{proposition}
\begin{proof} Let $C$ be as in the statement of the proposition, and assume that $F$ is its proper face. Without loss of generality assume that $0\in \relint F$. Let $E$ be the smallest exposed face of $C$ that contains $F$. If $E=F$, there is nothing to prove, so assume that $F\neq E$. Thus, $F \cap \relint E =\emptyset$.
	
For every $p\in \relint E$ we have $-\alpha p\notin E$ for all $\alpha>0$ (otherwise $(p,-\alpha p)\in C$, and by the definition of a face $[p,-\alpha p]\in F$, which is impossible due to $F\cap \relint E = \emptyset$). It follows that $-p\notin \mT(0;E)$.  
	
By the tangential exposure property, $-p\notin \mT(0;C)$, hence, $-p$ can be separated from $\mT(0;C)$: there exists some $g\neq 0$ such that 
$$
\langle g, -p\rangle >\sup_{v\in \mT(0;C)}\langle g,v\rangle = 0.
$$  
Observe that the normal $g$ defines a supporting hyperplane to $\mT(0;C)$ (and hence to $C$) that contains zero, but does not contain $E$ (since $\langle g,p\rangle <0$ for $p\in \relint E$). This supporting hyperplane exposes some face $G$ of $C$ which contains $F$, because $0\in \relint F$. The intersection $G\cap E$ is a nonempty face of $C$ that contains $F$. Since both $G$ and $E$ are exposed, their intersection is also exposed. The face $G\cap E$ is exposed, contains $F$ and is strictly smaller than $E$. This contradicts the definition of $E$.
\end{proof}

There are regular cones which are facially exposed, not FDC and not tangentially exposed.  The example from Roshchina \cite{Roshchina} satisfies
these properties, see Figure~\ref{fig:original}. Nevertheless, there are facially exposed regular cones that are also tangentially exposed, but not FDC.  We can prove this
by modifying the example from \cite{Roshchina}.

\begin{example}\label{eq:one} We revisit the example from \cite{Roshchina}. The closed convex cone $K\subset \R^4$ is a standard homogenization $K = \cone \{C\times \{1\}\}$ of a compact convex set $C\subset \R^3$ whose construction and Mathematica rendering are shown in Fig.~\ref{fig:original}.
\begin{figure}[ht]
	\centering
	\includegraphics[scale=1]{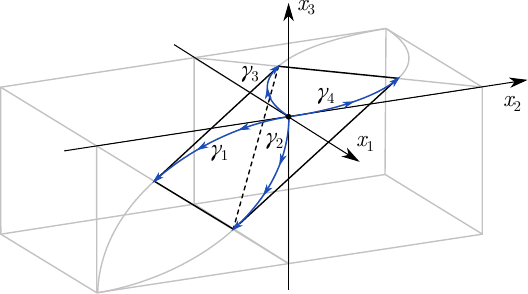}
	\qquad
	\includegraphics[height = 120pt]{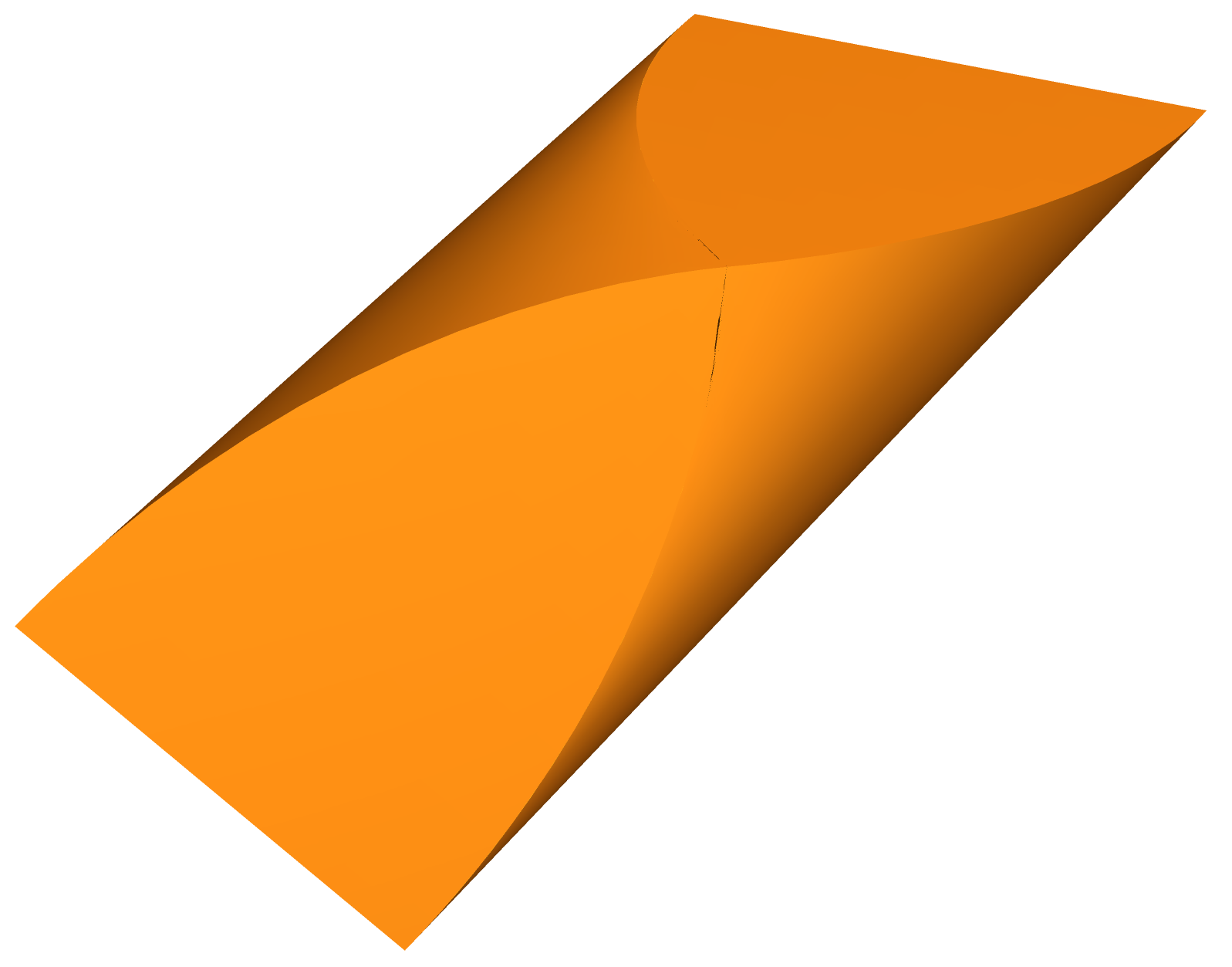}
	\caption{A slice of a closed convex cone that is facially exposed but not FDC. Notice that this set is not \emph{strongly facially exposed}
	(i.e., there exists at least a face that is not facially exposed).}\label{fig:original}
\end{figure}
The set $C$ is a nonsingular affine transformation of the convex hull of four curves.
In particular, it is $\conv \{\gamma_1,\gamma_2,\gamma_3, \gamma_4\}$, where
\begin{align*}\label{eq:curves-original}
\gamma_1 (t) & := \left(0, - \sin t,  \cos t -1\right),
&
\gamma_2 (t) & := \left(0, \cos t-1, - \sin t \right),\notag\\
\gamma_3 (t) & := \left(-\sin t, 1- \cos t,  0 \right),
&
\gamma_4 (t) & := \left(\cos t -1, \sin t, 0\right),
\end{align*}
and $t\in [0, \pi/4]$.
It is not difficult to observe that if $C$ fails the tangential exposure property, then its homogenization $K$ does as well (if the convex set $C$
is not tangentially exposed then the certificate of this fact---a face $F$ and $x \in F$---leads to a corresponding certificate for $K$ failing the tangential
exposure property). The failure of tangential exposure for the set $C$ is evident from considering tangents to the face $F = \conv \{\gamma_3,\gamma_4\}$ and $C$ at the point $(0,0,0)$. Indeed, it is clear that $g:= (0,-1,0)\in \mT(x;K)$ since
$$
(0,-1,0) = \Limsup_{t\to \infty} t\gamma_1(t^{-1}) = \lim_{s\downarrow 0} \frac{(0, -\sin s, \cos s -1)}{s}.
$$
On the other hand,
$$
\langle g, \gamma_3(t)\rangle = \cos t -1\leq 0, \qquad \langle g, \gamma_4(t)\rangle = -\sin t\leq 0 \quad \forall t\in [0,\pi/4],
$$
hence $g$ is separated strictly from $\mT(x;F)$. This is illustrated geometrically in Fig.~\ref{fig:original-tangent}.
\begin{figure}[ht]
	\centering
	\includegraphics[scale=1]{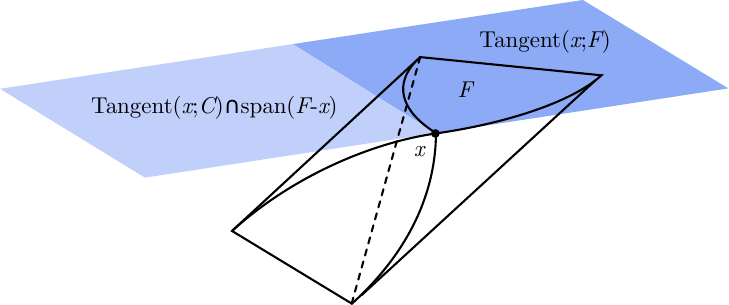}
	\caption{Failure of tangential exposure}\label{fig:original-tangent}
\end{figure}
\end{example}

\begin{example}\label{eg:cubic}
We construct a modified example of a closed convex cone that is facially \emph{and} tangentially exposed, but is not facially dual complete. This cone is a homogenization of the three-dimensional set $C$ that is a convex hull of two curves, one is a piece of a parabola, and the other one is a twisted cubic (see Fig.~\ref{fig:cubic}).
\begin{figure}[ht]
	\centering
	\includegraphics[scale=1]{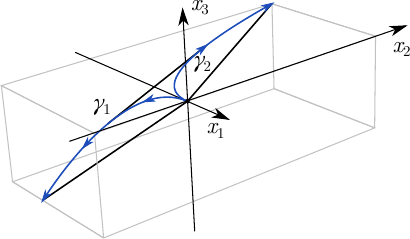}
	\includegraphics[height = 110pt]{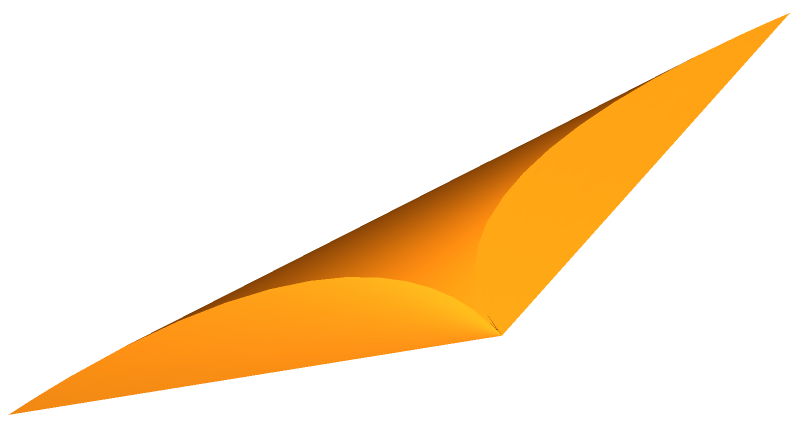}
	\caption{A rendering of construction of Example 2: A slice of a closed convex cone that is tangentially exposed but not facially dual complete.}\label{fig:cubic}
\end{figure}
So, we have $K := \cone \{C\times \{1\}\}$, $C := \conv\{\gamma_1,\gamma_2\}$, where
$$
\gamma_1(s)  = (-s, -s^2, -s^3), \, s\in [0,1]
 \quad  \text{and}\quad
\gamma_2(t)  = (-t, t^2, 0),\;t\in [0,1/3 (2  + \sqrt{7} )].
$$
It is a technical exercise to show that the cone $K$ (or equivalently the set $C$) is tangentially exposed, but not FDC.
\begin{figure}[ht]
	\centering
	\includegraphics[scale=1]{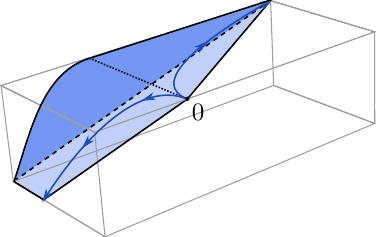}
	\includegraphics[height = 110pt]{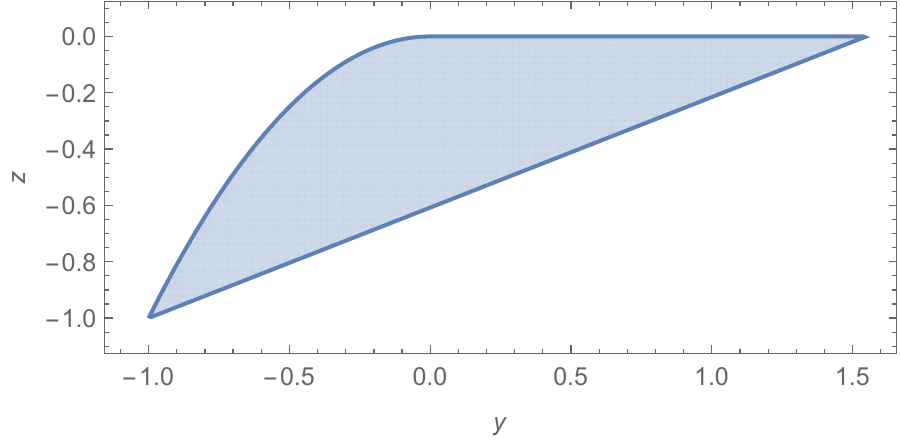}
	\caption{An illustration of how the tangent cone at the origin for Example 2 is not tangentially exposed.}\label{fig:second-order-cubic-fail}
\end{figure}
We leave the detailed algebraic computations, as well as the proof that the set is not FDC, to the Appendix.
\end{example}

\subsection{Lexicographic tangent cones}

The last example leads us to the next idea.  The above regular cone is facially exposed and tangentially exposed, but it is not FDC.  Also, its tangent
cone to $C$ at $x=(0,0,0)$ is not tangentially exposed itself. This is intuitively clear from  Fig.~\ref{fig:second-order-cubic-fail}, where the dotted line
in the left-hand-side graphic shows the set of points for which the tangential exposure property fails (on the tangent cone at $(0,0,0)$)
with respect to the adjacent flat face, and the right-hand-side plot shows the slice of this second-order tangent cone.  So, we consider
a stronger property defined by enforcing tangential exposure
condition \eqref{eq:GeomCond} recursively on all tangent cones. For example, a second-order tangent cone for $C$ at $x \in C$ and $v \in \mT(x;C)$ is:
\begin{eqnarray*}
\mT\left[v; \mT(x;C)\right] & = & \Limsup_{t_2 \to +\infty} t_2 \left[\mT(x;C)-v\right]\\
& = & \Limsup_{t_2 \to +\infty} t_2 \left\{\left[ \Limsup_{t_1 \to +\infty} t_1 (C-x)\right]-v\right\}.
\end{eqnarray*}
We may recursively apply this construction to generate $k$th-order tangent cones for every nonnegative
integer $k$.
This geometric notion is a geometric counterpart of
Nesterov's \emph{lexicographic derivatives} (see \cite{NesterovLex} for this analytic notion,
and the references therein).
Any tangent cone obtained as a result of the above recursive procedure (of any order) is called a \emph{lexicographic tangent cone} of $C$.
We say that a closed convex set is {\em strongly tangentially exposed} if
it is tangentially exposed along with all of its lexicographic tangent cones.

Next, we investigate some fundamental properties of the family of lexicographic tangent cones of closed convex sets.
Observe that for $u,v \in C$ such that $\face(u;C) = \face(v;C)=:F$, we have
\[
\mT(u;C) = \mT(v;C)=:\mT(F;C).
\]
That is, $\mT(F;C)$ denotes the tangent cone for $C$ at any $x \in \relint F$ for $F \unlhd C$.
Thus, the cardinality of distinct tangent cones of $C$ is bounded by the cardinality of the set of faces of $C$.
With this notation, our
Theorem \ref{thm:NecFDC} can be restated as:\\
Let $K$ be a regular cone that is FDC. Then for every pair of
faces $F,G$ such that $G \lhd F \unlhd K$, we have
\[
\mT(G;K) \cap \lspan F = \mT(G;F).
\]

Let
\[
\cT : \textup{ families of non-empty closed convex sets in $\R^n$ } \to \textup{families of non-empty closed convex cones in $\R^n$},
\]
defined by
\[
\cT(\cK) := \left\{\mT(F;K) : \,\, \forall F \unlhd K, \,\, F \neq \emptyset, \,\, \forall K \in \cK \right\},
\]
i.e. 
\[
\cT(\cK) = \textup{ the set of all tangent cones of convex sets in $\cK$.}
\]
We further define $\cT^0(\cK):=\cK$ and for every positive integer $k$, $\cT^k(\cK):= \cT\left[\cT^{k-1}(\cK)\right].$
Note that, if for some family of convex sets $\cK$, we have $\cT(\cK) = \cK$, then
\begin{equation}
\label{eqn:tandepth}
\cT^k(\cK) = \cK, \,\,\,\, \textup{ for every nonnegative integer } k.
\end{equation}
Let $C$ be a closed convex set. We abuse the notation slightly and write $\cT(C)$ for $\cT(\{C\})$ (when $\cK$ is a singleton $C$, we write $\cT^k(C)$ instead of
$\cT^k(\{C\})$). 
Then, the \emph{tangential depth} of $C$ is the smallest nonnegative integer
$k$ such that $\cT^{k+1}(C) = \cT^k (C).$ The tangential depth of $\R^n$ is zero for every nonnegative integer $n$ and the tangential
depth of $\R^n_+$ is one for every positive integer $n$.
For example, $\cT(\R_+) = \{\R_+, \R\} = \cT^2(\R_+)$, and, 
\[
\cT(\R_+^3) = \{\R_+^3, \R_+^2 \times \R, \R_+ \times \R^2, \R^3 \} = \cT^2(\R_+^3).
\]
In the above, we listed the elements of $\cT(\R_+^3)$ up to linear isomorphism (there are eight cones in $\cT(\R_+^3)$;
three of them are isomorphic to $\R_+^2 \times \R$, and another group of three are isomorphic to $\R_+ \times \R^2)$.
Next, for every positive integer $n$, consider the second order cone $\textup{SOC}^n$.
\[
\cT(\textup{SOC}^n) = \left\{\textup{SOC}^n, \textup{ a closed half space}, \R^n \right\} = \cT^2(\textup{SOC}^n).
\]
Thus, the tangential depth of $\textup{SOC}^n$ is one, for every positive integer $n$.
Note that for $n=1$, the first two elements listed in $\cT(\textup{SOC}^n)$ are linearly isomorphic, and for $n \geq 2$,
the second element represents infinitely many such cones (one for each extreme ray of $\textup{SOC}^n$).

We call a nonempty regular cone \emph{smooth} if every boundary point of $K$ is on an extreme ray of $K$
and the normal cone of $K$ at every extreme ray of $K$ has dimension one so that every extreme ray of $K$ is
exposed by a unique supporting hyperplane of $K$.
All smooth cones have tangential depth one.  Using the fact that almost all regular cones are smooth (in the space of
all regular cones),
we can conclude that almost all regular cones have tangential depth one.  Indeed, we must caution the reader that
this last statement is measure theoretic in nature and many of the interesting regular cones we encounter in optimization
are not smooth.

Given a nonempty closed convex cone $K$, suppose there exists a
nonnegative integer $k$ such that $\cT^{k+1}(K) \setminus \cT^k(K)$ contains only polyhedral cones and cones
$C$ with the property that when we express $C=\bar{C} + L$ with $L$ being the lineality space of $C$, the cone $\bar{C}$
is a smooth cone.  Then, using the above ideas, we can prove that the tangential depth of $K$ is at most $(k+2)$. 

Next, we prove that the tangential depth of every regular cone
is bounded by its dimension.

\begin{theorem}
	\label{thm:tangentialdepth}
Let $K \in \R^N$ be a nonempty closed convex cone. Then, the tangential depth of $K$ is at most $(d-\ell)$,
where $d$ is the dimension of $K$ and $\ell$ is the dimension of the lineality space of $K$.
	\end{theorem}

\begin{proof}
Let $K$ be as in the statement of the theorem and let $L$ denote the lineality space of $K$.  For every proper face $F \lhd K$, $\lspan(F) \supseteq L$.
If $\lspan(F) = L$, then $\mT(F;K) = K$.  However, if $\lspan(F) \setminus L \neq \emptyset$,
then since $\lspan(F)$ is a linear subspace, and $\mT(F;K)$ contains $\lspan(F)$, the dimension of
the lineality space of $\mT(F;K)$ is at least $(\ell+1)$.  Now, let $k$ be a nonnegative integer and apply this
observation to every cone in $\cT^k(K)$.  We conclude that every cone $K'$ in $\cT^{k+1}(K) \setminus
\cT^k(K)$ is $\mT(F;\tilde{K})$ for some parent cone $\tilde{K} \in \cT^k(K)$ and for a proper face $F$ of $\tilde{K}$.
Now, combining this with the observation \eqref{eqn:tandepth}, we see that for $k:=d-\ell$, $\cT^{k+1}(K) \setminus
\cT^k(K)=\emptyset$.  Therefore, the tangential depth of $K$ is at most $(d-\ell)$. 
\end{proof}

Therefore, a regular cone $K$ is strongly tangentially exposed iff every cone
in the set $\cT^d(K)$ is tangentially exposed, where $d:= \dim(K)$.
Our next goal is to prove that strongly tangentially exposed closed convex cones are FDC.

\subsection{Proof of the sufficient condition}\label{sec:SuffProof}

We use several technical claims in the proof.  The next proposition immediately
follows from the above definitions.

\begin{proposition}\label{prop:TangentInherits} Tangent cones inherit strong tangential exposure property from the original object. That is,
if $C$ is strongly tangentially exposed, then every $T \in \cT^k(C)$ is strongly tangentially exposed for every nonnegative integer $k$.
\end{proposition}

\begin{proposition}\label{prop:ExposePointed} Let $K$ be a regular cone in $\R^n$, and let $F\lhd K$ be an exposed face of $K$, $L := \lspan F$. Then for every nonzero $u\in F^* \cap L$ such that $u$ exposes $\{0\}$ as a face of $F$, there exists $g\in K^*$ such that $ u = \Pi_{L} g$.
\end{proposition}
\begin{proof}
Let $K,F,$ and $L$ be as above, and let $u\in F^* \cap L$ be such that $\iprod{u}{x} > 0,$ $\forall x \in F \setminus \{0\}$. Without loss of generality, we may assume $\|u\|=1$.
Since $F$ is an exposed proper face of $K$, there exists $s \in K^*$ such that
\[
\iprod{s}{x} \left\{\begin{array}{rl}
= 0, & \textup{ if } x \in F;\\
> 0, & \textup{ if } x \in K\setminus F.
\end{array}
\right.
\]
Let $g_{\alpha}:= u + \alpha s,$ $\alpha \in \R$.  If there exists $\alpha$ such that $g_{\alpha} \in K^*$, then
we are done.  So, we may assume that for every $\alpha \in \R$, there exists $x_{\alpha} \in K$ such that
\[
0 > \iprod{g_{\alpha}}{x_{\alpha}} = \iprod{u}{x_{\alpha}} + \alpha \iprod{s}{x_{\alpha}}.
\]
Since $K$ is a cone, we can choose $x_{\alpha}$ to be unit norm.  Now, as $\alpha \to +\infty$,
the sequence $\left\{x_{\alpha}\right\}$ must have a convergent subsequence with limit $\bar{x} \in K$ which also has norm 1.
If $\iprod{s}{\bar{x}} > 0$, then using
\[
-1 \leq -\|u\| \|x_{\alpha}\| \leq \iprod{u}{x_{\alpha}} < -\alpha \iprod{s}{x_{\alpha}}
\]
and taking limits as $\alpha \to +\infty$ along the subsequence of $\{x_{\alpha}\}$ converging to $\bar{x}$, we reach a contradiction.
Hence, we may assume $\iprod{s}{\bar{x}} = 0$, i.e., $\bar{x} \in F$. Applying the above limit argument with this new information, we conclude $\iprod{u}{\bar{x}} \leq 0$.
Thus, by our choice of $u$, $\bar{x} = 0$, again leading to a contradiction.  Therefore, there exists $\alpha$ such that $g_{\alpha} \in K^*$, and
we are done.
\end{proof}

Next, we observe that FDCness and strong tangential exposedness are not affected by addition or removal of subspaces.

\begin{proposition}\label{prop:ReduceDimension} Let $K = C+L$, where $L$ is a 
linear subspace and $C$ is a closed convex cone such that $\lspan C \subseteq  L^\perp$. Then the following statements are true.
	\begin{itemize}
		\item[(i)] The cone $K$ is strongly tangentially exposed if and only if $C$ is;
		\item[(ii)] The cone $K$ is FDC if and only if $C$ is.
	\end{itemize}
\end{proposition}
\begin{proof} For any $x\in K$ and its unique projection $p$ onto $C$ we have 
$$
\mT(x;K) = \mT(p;K); \quad \mT(x;E) = \mT(p;E)\quad \forall E \lhd K; 
$$	
moreover, observing that the faces of $C$ and $K$ are in bijective correspondence with each other ($F\lhd C$ if and only if $F+L \lhd K$), and that
\begin{align*}
\mT(x;K) & = \mT(p;C)+L,\\
\mT(x;F+L) & = \mT(p;F)+L\quad \forall F \lhd C,\\
\lspan (F+L) & = \lspan (F)+L \quad \forall F \lhd C,
\end{align*}
we obtain (i) directly from the definition of tangential exposure. 
	
Proof of (ii) likewise follows from the definitions and fundamental properties.
\end{proof}

Now, we are ready to prove our sufficient condition for FDCness.
\begin{theorem}[Sufficient condition]\label{thm:SuffFDC} If a closed convex cone $K\subseteq \R^n$ is strongly tangentially exposed, then it is facially dual complete.
\end{theorem}
\begin{proof} We will prove the statement by induction in the dimension $n$ of the underlying space $\R^n$. Observe that for $n=1$ the statement is trivial: all three possible, at most one-dimensional,
nonempty, closed convex cones are both strongly tangentially exposed and facially dual complete.

Assume now that every closed convex cone of dimension at most $(n-1)$ that is strongly tangentially exposed is also FDC.  We will prove the statement for $n$-dimensional closed convex cones.
Let $K\subseteq \R^n$ be a strongly tangentially exposed closed convex cone. To prove that $K$ is FDC, by Proposition~\ref{prop:CharNice} it suffices to show that
for all $F\lhd K$, with $L := \lspan F$, for every $u\in F^* \cap L$, we have $u\in \Pi_{L}K^*$.

Let $u \in F^* \cap L$, we may assume $u$ is not zero, and define
\[
E:= \left\{x \in F: \,\, \iprod{u}{x}=0 \right\}.
\]
Observe that $E\lhd F\lhd K$, since $u$ defines a supporting hyperplane to $F$ at origin, and any sub-face of a face is also a face (see Proposition~\ref{prop:ff}), if $E=\{0\}$, the result follows from Proposition~\ref{prop:ExposePointed}. Otherwise $\dim E \geq 1$. Let $x\in \relint E$ and consider $\mT(x;K)$ and $\mT(x;F)$. Observe that $\lspan E \subset \mT(x;F) \subset \mT(x;K)$, so that our cones decompose into a direct sum:
$$
\mT(x;K) = C+\lspan E,
$$
where $C \subseteq (\lspan E)^\perp$. Notice that since $\dim E\geq 1$, we have $\dim C \leq n-1$.
	
By Proposition~\ref{prop:TangentInherits}, the cone $\mT(x;K)$ inherits strong tangential exposedness property from $K$. Applying Proposition~\ref{prop:ReduceDimension} (i) to $\mT(x;K)$ and $C$, we deduce that $C$ is strongly tangentially exposed as well, and since the dimension of $C$ is less than $n$, it is FDC by the induction hypothesis. Applying Proposition~\ref{prop:ReduceDimension} (ii) to $\mT(x;K)$ and $C$, we deduce that $\mT(x;K)$ is facially dual complete. 

We consider two cases based on whether $\mT(x;F)$ is a face of $\mT(x;K)$ or not.

\underline{Case 1:}
$\mT(x;F)$ is a face of $\mT(x;K)$. Then from the FDCness of $\mT(x;K)$ there exists $g\in (\mT(x;K))^* \subset K^*$ such that with $L= \lspan \mT(x;F) = \lspan F$, $u = \Proj_{L} g$, and we are done.
	
\underline{Case 2:}
$\mT(x;F)$ is not a face of $\mT(x;K)$. Then consider the minimal face $G\lhd \mT(x;K)$ that contains $\mT(x;F)$. By the property of minimal faces in Proposition~\ref{prop:ff}~(iii) we have
\[
\relint \left[\mT(x;F)\right] \cap \relint G\neq \emptyset,
\]
and therefore
\[
\left\{\relint \lspan \left[ \mT(x;F)\right]\right\} \cap \relint G \neq \emptyset.
\]
Applying Proposition~\ref{prop:NormalIntersection} to $\left[\lspan \mT(x;F)\right]$ and $G$, we have
\begin{equation}\label{eq:NormalRelation}
\left\{\left[\lspan \mT(x;F)\right]\cap G\right\}^* = G^* +\left[\mT(x;F) \right]^{\perp}.
\end{equation}
From the strong tangential exposure assumption we have
\[
\mT(x;F) = \mT(x;K)\cap \lspan \mT(x;F),
\]
and since $\mT(x;F)\subseteq G \subseteq \mT(x;K)$, this yields
\begin{equation}\label{eq:TangSubFace}
\left[\lspan \mT(x;F) \right] \cap G = \mT(x;F).
\end{equation}
From \eqref{eq:NormalRelation} and \eqref{eq:TangSubFace} we have:
\begin{equation}\label{eq:TangSum}
\left[\mT(x;F)\right]^*  = G^* + \left[\mT(x;F)\right]^\perp.
\end{equation}
Furthermore, since $G^*$ is closed, and $[\lspan G]^\perp \subset G^*$,  we have 
$$
G^* = G^*\cap \lspan G + [\lspan G]^\perp.
$$
Using this observation together with  $[\lspan G]^\perp \subseteq [\mT(x;F)]^\perp$, we obtain from \eqref{eq:TangSum}
$$
\left[\mT(x;F)\right]^* = G^*\cap \lspan G + \left[\mT(x;F)\right]^\perp.
$$
By our choice of $x$ we have $u\in [\mT(x;F)]^*$, hence, $u$ is the orthogonal projection of some $g\in G^* \cap \lspan G$ onto $\lspan \mT(x;F)$.

Since $G$ is a face of $\mT(x;K)$, and $\mT(x;K)$ is FDC, we can now find a point $g'$ in $(\mT(x;K))^* \subset K^*$ that projects onto $\lspan G$ as $g$. 

Now $g$ is the orthogonal projection of $g'\in K^*$ onto $\lspan G$, and $u$ is the orthogonal projection of $g$ onto $\lspan F \subseteq \lspan G$. Hence $u = \Pi_{\lspan F}(g')\in \Pi_{\lspan F} K^*$.
\end{proof}

The sufficient condition for FDCness is not necessary, as is evident from the next example.

\begin{example}\label{eg:lost-tangent} Let $K= \cone\{C\times \{1\}\}\subset \R^4$, where $C\subset \R^3$ is a closed convex set, $C := \conv\{\gamma_1,\gamma_2\}$,
$$
\gamma_1(t) = (\cos t, \sin t, 1), \; t\in [0,\pi/2], \qquad \gamma_2(t) = (\cos t, \sin t, -1)\; t\in [0,\pi].
$$
The set $C$ is shown in Fig.~\ref{fig:facialexp-not-preserved}.
\begin{figure}[ht]
	\centering
	\includegraphics[scale=1]{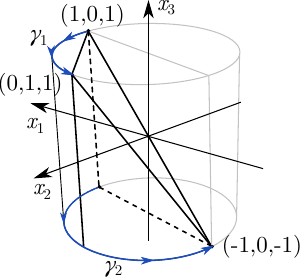}
	\qquad \qquad
	\includegraphics[height = 120pt]{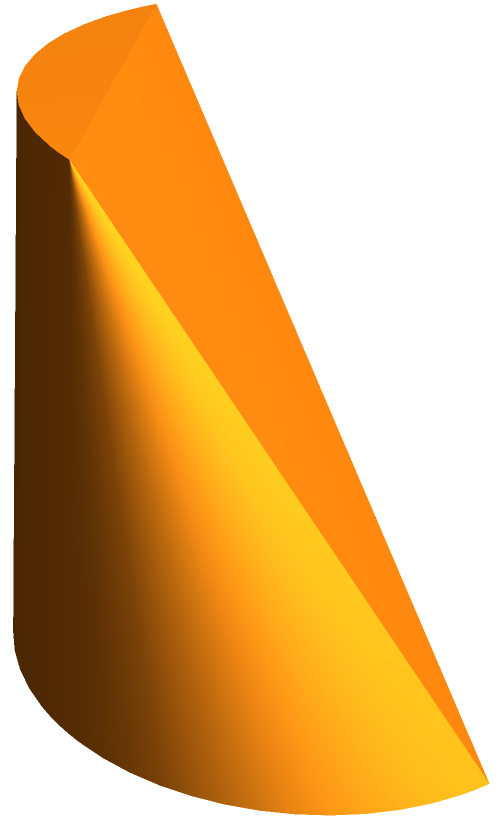}
	\caption{Construction of Example 3: A facially exposed set may have a tangent that is not facially exposed}\label{fig:facialexp-not-preserved}
\end{figure}
Observe that the set $C$ is tangentially (and facially) exposed. However, strong tangential exposure fails for this set. In particular,  $\mT(\bar x;C)$, where $\bar x = (0,1,1)$ is not facially exposed (see its Mathematica rendering in the first image of Fig.~\ref{fig:tangent-projections}),
\begin{figure}[ht]
	\centering
	\includegraphics[height=120pt]{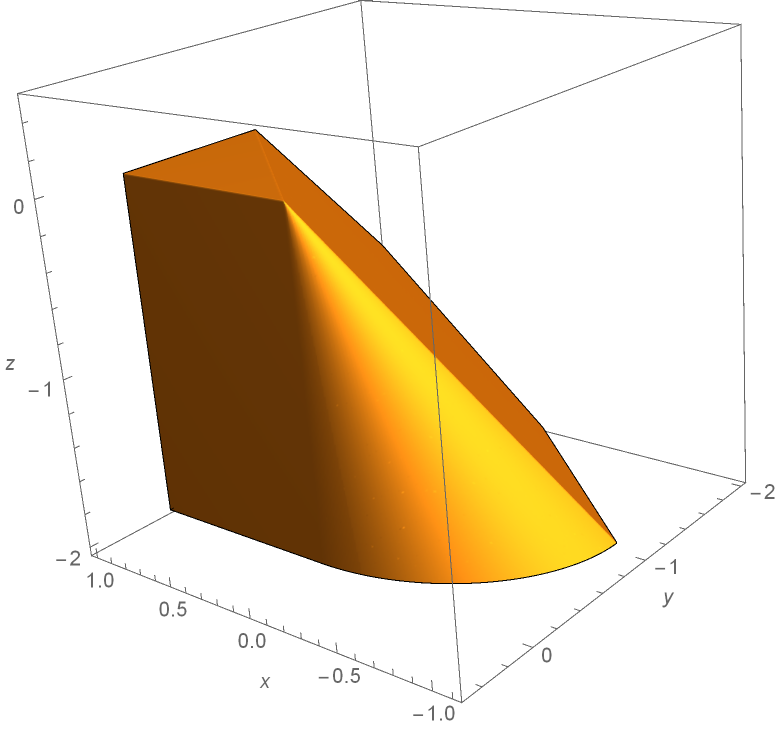}
	\qquad
	\includegraphics[height = 120pt]{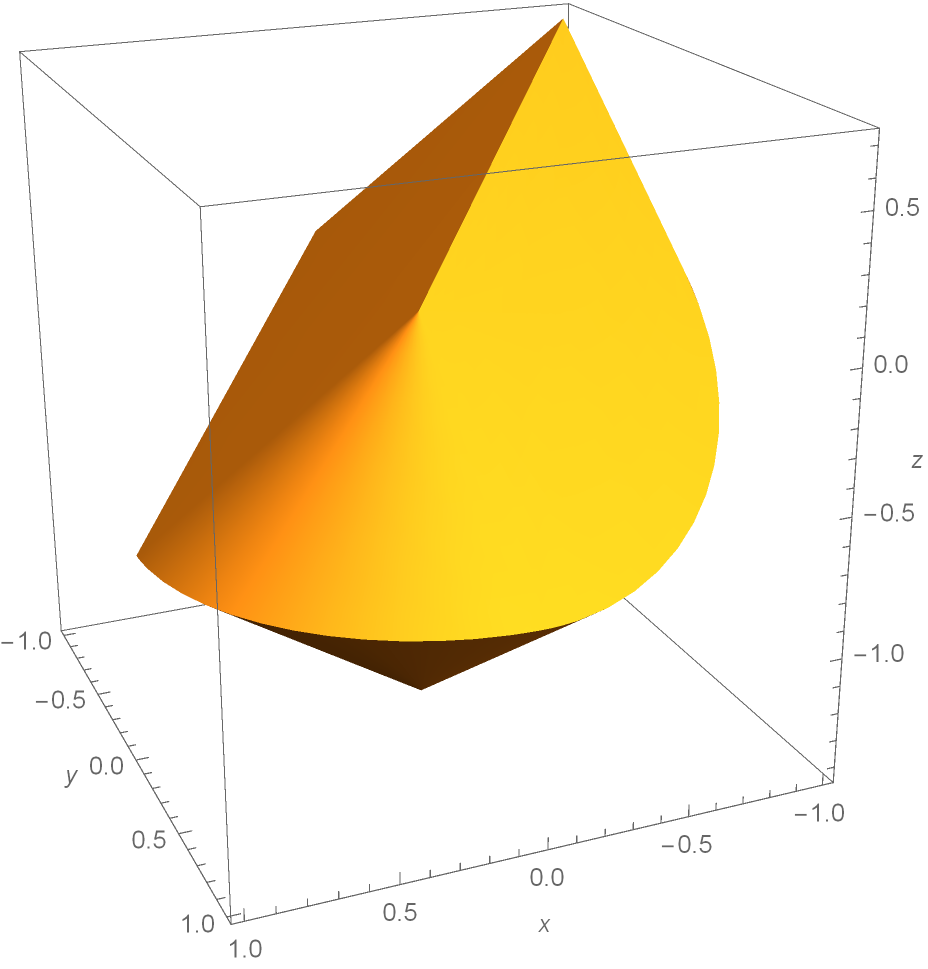}
	\qquad
	\includegraphics[height = 120pt]{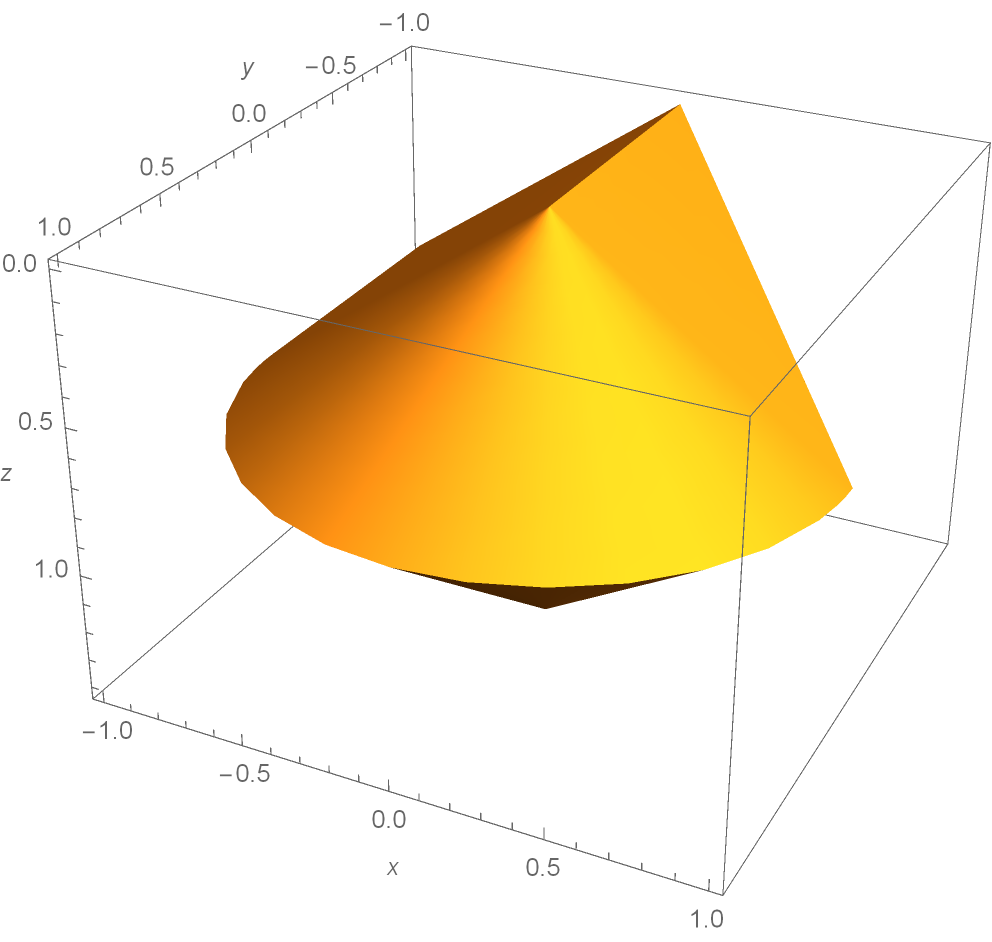}
	\caption{Tangent cone of the cone from Example 3 at $\bar{x}:=(0,1,1)$.  This tangent cone is not facially exposed and the right-most pictures
	illustrate two closed convex sets whose conic hulls represent the projections of the dual cones on the relevant subspaces.}\label{fig:tangent-projections}
\end{figure}
and hence it is not tangentially exposed either. At the same time this cone is facially dual complete. In this case we only need to check the identity $\Pi_{\lspan F}(F^\perp +K^*) = F^* \cap \lspan F$ for the faces of $K$ that correspond to the top and bottom faces of $C$, and for both cases the relevant projections are the conic hulls of three dimensional sets shown in the last two images in Fig.~\ref{fig:tangent-projections}. We provide all relevant technical computations in the Appendix.
\end{example}

\section{Conclusion}

We provided tighter, geometric, primal characterizations of facial dual completeness of regular convex cones
via tangential exposure property and strong tangential exposure property.  In Figure~\ref{fig:Tablet2} we
present a schematic summary of our results.  Each bubble in the figure corresponds to a property
of convex cones (facial exposedness, facial dual completeness, etc.).  A solid arrow from one bubble
to another bubble illustrates the fact that the former property implies the latter (labels on solid arrows indicate
where such a result was proved first; if the implication is trivial, the solid arrow has no label).
A dashed arrow which is blocked indicates that proving the underlying implication is impossible
(dashed, blocked arrows are labeled by a corresponding example proving this claim).

Our results provide geometric
tools for checking FDCness directly on the primal cone.  However, we do not provide any provably efficient
algorithmic tools for checking these properties.  A related problem is whether  Ramana's Extended Lagrange-Slater
Dual (ELSD) construction \cite{Ramana1997} can be extended to tangentially exposed cones.  Some sufficient
conditions for generalizing this construction were discussed in \cite{TW2012} and a geometric extension of ELSD
to FDC cones was established in \cite{Pataki2013}.  The cone of positive semidefinite matrices
as well as any regular convex cone that can be expressed as the intersection of some positive semidefinite cone and a
linear subspace is strongly tangentially exposed.  Also, there are strongly tangentially exposed regular convex cones that
are not semi-algebraic sets.  The problems of characterizing the set of tangentially exposed convex cones and
characterizing the set
of strongly tangentially exposed convex cones are left for future research.

As a by-product of our approach, we have introduced some new notions of exposure
for faces of closed convex sets:
\begin{enumerate}[(i)]
\item
tangentially exposed convex sets
\item
convex sets with facially exposed tangent cones
\item
convex sets with every lexicographic tangent cone facially exposed
\item
strongly tangentially exposed convex sets.
\end{enumerate}

We can also apply these notions to the polars of convex sets.  Also, we can
ask for characterizations of closed convex sets $C$ such that $C$ and $C^{\circ}$ have
a specific property (or a specific pair of the properties) from the above list.

\begin{figure}[ht]
	\centering
\begin{overpic}[scale=1
]{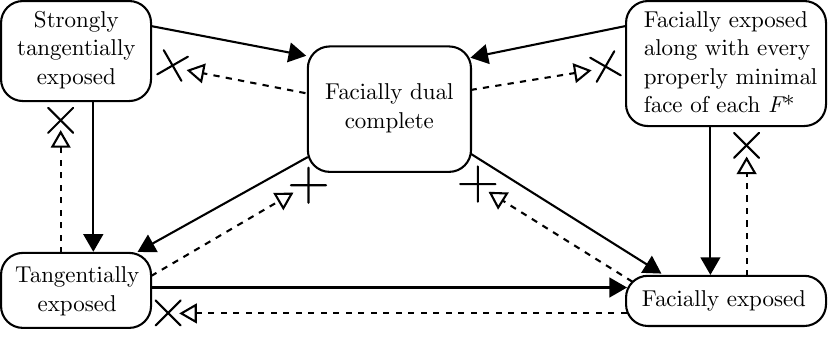}
 \put (15,19) {Theorem~\ref{thm:NecFDC}}
 \put (25,39) {Theorem~\ref{thm:SuffFDC}}
 \put (60,39) {Pataki \cite{PatakiFexpNice}}
 \put (66,19) {Pataki \cite{PatakiFexpNice}}
 \put (42,1) {Example~\ref{eq:one}}
 \put (51,13) {Example~\ref{eq:one}}
 \put (3,13) {\rotatebox{90}{Example~\ref{eg:cubic}}}
 \put (21,28) {Example~\ref{eg:cubic}}
 \put (31,13) {Example~\ref{eg:lost-tangent}}
 \put (60,28) {Example~\ref{eg:lost-tangent}}
 \put (40,9) {Proposition~\ref{prop:texpexp}}
 \put (93,11) {\rotatebox{90}{Example~\ref{eg:lost-tangent}}}
\end{overpic}
	\caption{A schematic summary of main results of this paper and their relation to other prior results.}
	\label{fig:Tablet2}
\end{figure}
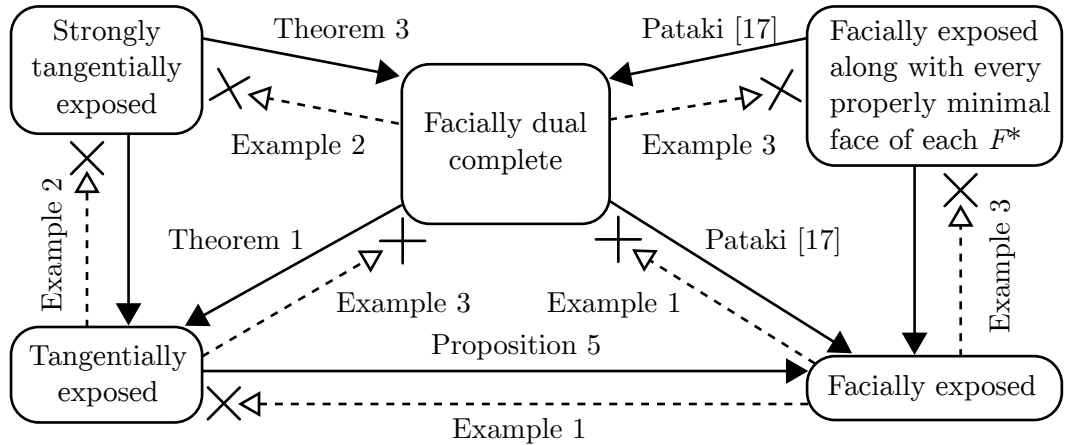

\bibliographystyle{plain}
\bibliography{refs}

\section*{Appendix}

The goal of this section is to demonstrate that the cones in Examples~\ref{eg:cubic} and \ref{eg:lost-tangent} satisfy the claimed properties. We use a substantial number of technical results which are listed below and precede the main statements (Propositions~\ref{prop:eg2} and~\ref{prop:eg3}). In some of the proofs we only provide the ideas behind the computations, so that the tedious technical details can be reconstructed using the basic tools of linear algebra and real analysis.

\begin{proposition}\label{prop:expint} Suppose that $E = F\cap G$, where $F$ and $G$ are exposed faces of a closed convex set $C\subset \R^n$. Then $E$ is an exposed face of $C$.
\end{proposition}
\begin{proof} Since both $F$ and $G$ are exposed, there exist $p_F,p_G\in \R^n$ such that
$$
\Argmax_{x\in C}\langle p_F, x \rangle = F, \qquad \Argmax_{x\in C}\langle p_G, x \rangle = G.
$$
Denote
$$
m_F := \max_{x\in C}\langle p_F, x \rangle, \qquad m_G := \max_{x\in C}\langle p_G, x \rangle.
$$
Let $p_E:= p_F+p_G$. We have
$$
\langle p_E, x \rangle = \langle p_F, x \rangle+ \langle p_G, x \rangle< m_F+m_G \quad \forall x\in C\setminus (F\cap G);
$$
$$
\langle p_E, x \rangle = \langle p_F, x \rangle+ \langle p_G, x \rangle= m_F+m_G \quad \forall x\in E= F\cap G.
$$
Hence,
$$
\Argmax_{x\in C}\langle p_E,x\rangle = E,
$$
and therefore $E$ is an exposed face of $C$.
\end{proof}

\begin{proposition}\label{prop:boundary} Let $C$ be a compact convex set with a nonempty interior, and let $\mathcal H$ be a collection of half-spaces that contain $C$. If for every point on the boundary of $C$ there is at least one half-space $H\in \mathcal{H}$ whose boundary hyperplane contains this point, then
$$
C = \bigcap_{H\in \mathcal{H}} H.
$$
\end{proposition}
\begin{proof} Assume the contrary, i.e. the conditions of the proposition are satisfied, but there is a point $x\in (\bigcap_{H\in \mathcal{H}} H)\setminus C$. Since $\inte C\neq \emptyset$, there is some $y\in \inte C$. The line segment $[x,y]$ intersects the boundary of $C$ at a unique point $z\in (x,y)$ (see \cite[Remark 2.1.7]{JBHU}). For some $H\in \mathcal{H}$ there is a boundary hyperplane that contains $z$. The half-space must have $y$ in its interior, hence $x\notin H$, and therefore $x\notin \bigcap_{H\in \mathcal{H}} H$, a contradiction.
\end{proof}

\begin{proposition}\label{prop:relint} Let $\mathcal{F}$ be a collection of proper faces of a compact convex set $C\subset \R^3$, $\inte C \neq \emptyset$. If there exists a homeomorphism $\phi$ from the union $U$ of the relative interiors of the sets in $\mathcal F$,
$$
U = \bigcup_{F\in \mathcal{F}} \relint F
$$ 
to the Euclidean sphere $S_2$, then the collection $\mathcal{F}$ contains all nonempty proper faces of $C$. 
\end{proposition}
\begin{proof}
It is not difficult to construct a homeomorphism $\psi$ between the boundary of $C$ and the unit sphere. This can be done by choosing an arbitrary point $c\in \inte C$ and identifying each point $u$ on the boundary of $C$  with the point $p = (u-c)/\|u-c\|$. This mapping is continuous, and since the intersection of the ray $c+\cone p$ with the boundary of $C$ is unique (see \cite[Remark~2.1.7]{JBHU}), it is also a bijection, hence the mapping $\psi$ is indeed a homeomorphism. 

We can compose the inverse of the homeomorphism $\phi$ (from the assumption) with $\psi$ to obtain another homeomorphism $\psi \circ \phi^{-1}$ that maps the unit sphere to its subset. If there exists a point on the boundary of $C$ that is not in $U$, then the set 
$$
\psi (\phi^{-1}(S_2))
$$ 
is a proper subset of the sphere. This is impossible by the standard argument involving the stereographic projection and Borsuk-Ulam Theorem: if such homeomorphism existed, it is easy to construct another homeomorphism between the sphere and the Euclidean subspace of the same dimension by rotating the sphere and considering the stereographic projection. Being a homeomorphism, this is a continuous mapping, which by Borsuk-Ulam Theorem has to have coincident images of two antipodal points.
\end{proof}

\begin{proposition}\label{prop:explift} Let $C$ be a compact convex set in $\R^n$ and let $K$ be its lifting to $\R^{n+1}$,  $K: = \cone \{C\times \{1\}\}$. The set $C$ is facially (tangentially) exposed if and only if $K$ is.
\end{proposition}
\begin{proof} The facial exposure part was proven in \cite[Proposition 3.2]{Roshchina}. The tangential exposure can be shown in a similar fashion, using the face correspondence given in \cite[Proposition 3.1]{Roshchina}.
\end{proof}

\begin{proposition}\label{prop:twodim} If a closed convex set $C\subset \R^n$ is facially exposed, then all zero- and one-dimensional faces of $C$ are tangentially exposed, i.e.
\begin{equation}\label{eq:TExp01}
\lspan (F-x)\cap \mT(x;C) = \mT(x;F)\quad \forall \, x\in F, \quad \forall\, F, \quad \dim F<2.
\end{equation}
\end{proposition}
\begin{proof} Observe that all zero-dimensional faces are tangentially exposed due to the triviality of the relevant linear span, so we only need to prove the statement for one-dimensional faces.

Assume that there exists a face $[u,v]$, $u\neq v$ of a closed facially exposed set $C$ such that $[u,v]$ is not tangentially exposed.

This means that there exists $x\in [u,v]$ that violates \eqref{eq:TExp01}. Observe that $x\notin (u,v)$, as for the points in the relative interior of the interval we have $\mT(x;[u,v]) = \lspan(u-x)$, and property \eqref{eq:TExp01} holds trivially. Without loss of generality we assume that $x=u$.

There exists a sequence $\{x_k\}$ such that $x_k\to u$, $x_k\in C$,
$$
p_k:= \frac{x_k-u}{\|x_k-u\|} \to p \in (\mT(x;C)\cap \lspan \{v-u\} )\setminus \mT(u;F).
$$
Observe that from $p\notin \mT(u;F) = \cone \{v-u\}$, $p\in \lspan \{v-u\}$, $\|p\|=1$ we deduce that
$$
p = \frac{u-v}{\|u-v\|}.
$$
Since $\{u\}$ is an exposed face of $C$, there exists a normal $q\in \R^n$ such that
$$
\langle q, u\rangle > \langle q, x\rangle \quad \forall x\in C.
$$
We therefore have
$$
\langle q, p \rangle = \lim_{k\to \infty } \frac{\langle q, x_k-u\rangle }{\|x_k-u\|} \leq 0,
$$
and on the other hand
$$
\langle q, p \rangle = \frac{\langle q, u-v\rangle }{\|u-v\|}>0,
$$
a contradiction.
\end{proof}

\begin{proposition} Let $F$ be a two-dimensional face of a three-dimensional compact convex set $C$. If for each $x\in F$ and each $q\in \mN(x;F)\cap \lspan (F-x)$ there exists a corresponding normal $h\in \mN(x;C)$ that projects onto the linear span of $F-x$ as $q$, then $F$ is tangentially exposed.
\end{proposition}
\begin{proof}
Suppose that $F$ is not tangentially exposed. This implies that there exists $x\in F$ and a sequence  $\{x_k\}$, $x_k\to x$, $x_k\in C$ such that
$$
p_k = \frac{x_k-x}{\|x_k-x\|}\to p\in (\mT(x;C)\cap \lspan (F-x) )\setminus \mT(x;F).
$$

Since $p\in\lspan (F-x)\setminus \mT(x;F)$, there must be a normal $q\in \mN(x;F)\cap \lspan (F-x)$ such that $\langle p,q\rangle <0$.

If there is a normal $h\in \mN(x;C)$ such that
$$
\Pi_{\lspan F}(h) = q,
$$
then for sufficiently large $k$
$$
\langle x_k-x,h\rangle <0,
$$
which is impossible.
\end{proof}

\begin{proposition}\label{prop:polarlift} Given the representation for our set $C$ as
$$
C = \{ \bar{x} \, : \, \langle p_t, \bar x\rangle\leq d_t, t\in T\},
$$
its lifting is
$$
K = \{ x \, : \, \langle (p_t,-d_t), x\rangle\leq 0, t\in T\},
$$
and the dual cone of the lifting is
$$
K^* = \cl\cone \{(p_t,-d_t)\, :\, t\in T\}.
$$
\end{proposition}
\begin{proof} Straightforward from the definitions.
\end{proof}

\begin{proposition}\label{prop:projection}
Let $L$ be a linear subspace and let $C$ be a closed convex set. The set $L^\perp+C$ is closed iff the projection of $C$ onto $L$ is closed. 
\end{proposition}
\begin{proof} First assume that $\Proj_{L}(C)$ is closed. Consider any sequence $\{x_k\}$ such that $x_k \in (L^\perp+C)$ for all $k\in \N$ and  $x_k \to \bar x$. Then $\Pi_L(x_k) \to \Pi_L(\bar x) \in \Pi_L(C)$ by our assumption. Hence there exists $\bar y\in C$ such that $\Pi_L(\bar x) = \Pi_L(\bar y) $.
We have 
$$
\bar x 
= \Pi_L(\bar x) + (\bar x-\Pi_L(\bar x)) 
= \Pi_L(\bar y) + (\bar x-\Pi_L(\bar x)) 
= \bar y+\underbrace{(\Pi_L(\bar y)-\bar y)}_{\in L^\perp}+\underbrace{(\bar x-\Pi_L(\bar x))}_{\in L^\perp},
$$
hence, $\bar x \in C+L^\perp$.

Now assume that $C+L^\perp$ is closed and let $\{x_k\}$ be such that $x_k\in \Pi_L(C)$ for all $k\in \N$ and $x_k\to \bar x$. For every $k\in \N$ there is some $y_k\in C$ such that $x_k = \Pi_L(y_k)$. We hence have
$$
x_k = y_k + (x_k-y_k) = y_k + (\Pi_L(y_k)-y_k) \in C+L^\perp.
$$
Since $C+L^\perp$ is closed, we have $\bar x = \bar y + \bar z$ with  $\bar y \in C$, $\bar z\in L^\perp$. Then $\bar x = \Pi_L(\bar y) \in \Pi_L(C)$, so $\Pi_L(C)$ is closed.  
\end{proof}

\begin{proposition}\label{prop:no-need-to-check} Let $K\subseteq \R^n$ be a cone, and assume that $K$ is facially exposed. Then for every $F\lhd K$ such that $F = \cone\{p_1,p_2\}$, where $p_1,p_2\in \R^n$ are linearly independent, the set $K^*+F^\perp$ is closed.
\end{proposition}
\begin{proof} Since $K$ is facially exposed, the faces $E_1 = F\cap \lspan p_1$ and $E_2 = F\cap \lspan p_2$ are exposed. Therefore, there are normals  $h_1,h_2\in \R^n$ such that
\begin{equation}\label{eq:normals005}
\langle h_i, p_i\rangle = 0,\qquad \langle h_i, x\rangle <0 \quad \forall x\in K\setminus E_i, \; i\in \{1,2\}.
\end{equation}
Observe that $h_1,h_2\notin F^\perp$ (since they expose proper faces of $F$). Hence,
$$
g_i:= \Pi_{\lspan F} (h_i)\neq 0 \qquad \forall i \in \{1,2\}.
$$
Moreover,
\begin{equation}\label{eq:equalityzero}
\langle g_i,p_i\rangle =  \langle g_i-h_i,p_i\rangle+\langle h_i,p_i\rangle = 0\; \forall i \in \{1,2\},
\end{equation}
since $g_i-h_i\in F^\perp$, and 
$$
\langle g_i,x\rangle = \langle h_i,x\rangle<0 \quad \forall x\in F\setminus E_i, i\in \{1,2\} .
$$
Observe that an $x\in \lspan F$ can be represented as
$$
x = \alpha p_1+\beta p_2, \qquad \alpha, \beta\in \R,
$$
with $\alpha,\beta\geq 0$ if and only if $x\in F$.
We have from \eqref{eq:equalityzero}
$$
\langle x, g_1\rangle = \alpha \langle p_1, g_1\rangle + \beta \langle p_2, g_1\rangle = \beta \langle p_2, g_1\rangle,
\qquad 
\langle x, g_2\rangle = \alpha \langle p_1, g_2\rangle + \beta \langle p_2, g_2\rangle=\alpha  \langle p_1, g_2\rangle.
$$
It follows from these relations that $\alpha\geq 0$ if and only if $\langle x, g_1\rangle \leq 0$ and $\beta \geq 0 $ if and only if $\langle x, g_2\rangle \leq 0$.
We have the representation
$$
F = \{x\in \R^n \,:\, \langle x, g_1\rangle \leq 0, \; \langle x, g_2\rangle \leq 0\}\cap \lspan F.
$$
For the dual face we have
$$
F^* = -\cl\cone \{g_1,g_2\}+F^\perp = -\cone \{g_1, g_2\}+F^\perp,
$$
hence, for any $y\in F^* $ we have
$$
y = -\alpha g_1-\beta g_2 + u,
$$
where $\alpha,\beta \in \R_+$ and $u\in F^\perp$.
We can rewrite this as
$$
y = -\alpha g_1-\beta g_2 + u = -\alpha h_1-\beta h_2 + (\alpha (h_1-g_1)+\beta(h_2-g_2)+u),
$$
where $\alpha (h_1-g_1)+\beta(h_2-g_2)+u\in F^\perp$, and since $h_1,h_2\in -K^*$, we have  $y\in K^*+F^\perp $. By the arbitrariness of $y$ this yields $F^*\subset K^* +F^\perp$. Together with $F^* = \cl (K^* +F^\perp)$ this yields $ K^* +F^\perp = \cl (K^* +F^\perp)$.

\end{proof}

\begin{proposition}[Pataki criterion]\label{prop:Pataki} If a face $F\lhd K$ is such that all proper minimal faces of $F^*$ are exposed, then $F^\perp+K^*$ is closed. 
\end{proposition}
\begin{proof}
This follows directly from Theorem~2 and the proof of Theorem~3 in \cite{PatakiFexpNice}.
\end{proof}

\begin{proposition}\label{prop:conichull} Let $S\subset \R^n$ be such that $S$ is compact and can be strictly separated from zero. Then $\cone S $ is a closed convex cone.
\end{proposition}
\begin{proof} If $\cone S$ is not closed, then there must be a sequence $\{y_k\}$ such that $y_k\in K$ for all $k\in \N$ and $y_k \to y \notin K$. Therefore for each $k\in \N$ we have 
$$
y_k = \sum_{i =1}^{p_k} \alpha_k^i x_k^i, \quad \sum_{i=1}^{p_k}\alpha^i_k =1,\quad  \alpha_k^i \geq 0 \;\forall i \in \{1, \dots, p_k\}, \quad p_k \leq n+1. 
$$ 
\end{proof}

\begin{proposition}[Properties of the cone $K$ from Example~\ref{eg:cubic}]\label{prop:eg2} Let $K := \cone \{C\times \{1\}\}$, where $C := \conv\{\gamma_1,\gamma_2\}$, $\gamma_1(s)  = (-s, -s^2, -s^3), \, s\in [0,1]$ and $\gamma_2(t)  = (-t, t^2, 0),\;t\in [0,1/3 (2  + \sqrt{7} )]$. The closed convex cone $K$ is
\begin{itemize}
\item facially exposed;
\item tangentially exposed;
\item not strongly tangentially exposed;
\item not FDC.
\end{itemize}
\end{proposition}
\begin{proof} To verify that $K$ is facially and tangentially exposed by Proposition~\ref{prop:explift} it is sufficient  to show that $C$ satisfies these properties. 

\textbf{To show facial exposure}, first consider the  parametric families of compact convex sets
$$
F_{11}(s) = [0,\gamma_1(s)], \; s\in (0,1], \quad F_{22}(s) = [\gamma_1(s),\gamma_2(\varphi(s))], \, s\in (0,1],
$$
where $\varphi(s) = 1/3 (2  + \sqrt{7} )s$, and 
$$
F_1 =  \conv \{0, \gamma_1(1), \gamma_2(\varphi(1))\}, \qquad  F_2 = \conv \{\gamma_2\}.
$$
To show that these sets are exposed one- and two-dimensional faces of $C$,  it is sufficient to demonstrate that for each of these faces there exists a corresponding exposing hyperplane. This is a straightforward exercise in analysis, which we omit for brevity.

It is evident that $ \gamma_1\cup \gamma_2 \subseteq \ext C$, since all points in $\gamma_1\cup \gamma_2$ are subfaces of the higher dimensional faces listed above. All these zero-dimensional faces are exposed by Proposition~\ref{prop:expint}.
 
It is evident from the diagram in Fig.~\ref{fig:sphere} that the relative interiors of all faces that we came across so far can be mapped homeomorphically to a sphere,
\begin{figure}[ht]
{\centering
\includegraphics[scale=1]{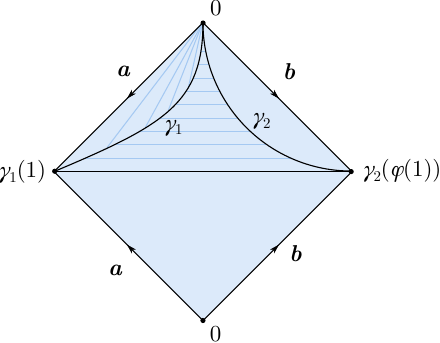}\\
}
\caption{Boundary of $C$  identified with the unit sphere}
\label{fig:sphere}
\end{figure}
therefore, by Proposition~\ref{prop:relint}, there are no proper faces of the set $C$ other than the listed exposed faces.

\textbf{Tangential exposure} needs to be verified for two-dimensional faces only due to Proposition~\ref{prop:twodim}. We only have two such faces, $F_1$ and $F_2$.

For the triangular face $F_{1}$ observe that all of its one-dimensional faces are exposed, hence the relevant normals project onto the normals at the points on these faces in the two-dimensional span of the face. The normals at the corner points are obtained as the convex hulls of these projections.

For the top face $F_2 = \conv \gamma_2$ the selection of the normals and the verification of the projections is a straightforward technical exercise.

\textbf{To show that the second-order tangential exposure is broken}  (and in fact the tangent cone is not even facially exposed), consider the tangent to the set $C$ at $0$. We have
$$
\mT(0;C) = \Limsup_{t\to \infty} t C =\cl\cone \{\gamma_1\cup\gamma_2\}.
$$
We scale our curves for convenience to obtain
$$
\kappa_1(s) = (-1, -s, -s^2),  \quad \kappa_2(t) = (-1,t,0).
$$

We hence have a slice of our tangent cone given by
$$
\conv \{(-s,-s^2), s\in [0,1], (-1,t,0), t\in [0, \varphi(1)]\},
$$
see Fig.~\ref{fig:second-order-cubic-fail}. It is clear that the set has an unexposed face $\{(0,0)\}$.

\textbf{To show that the cone $K= \cone \{C\times \{1\}\}$ is not FDC}, we explicitly identify a parametrised family of points in the sum $K^*+F^\perp$ whose limit does not belong to this set.  Let
$$
p(s) = \left(2(\sqrt{7}+1)s, (5-\sqrt{7}),0, (\sqrt{7}+3)s^2 \right).
$$
We will show that $p(s) \in K^* + F^\perp$ for $F=\cone \{F_2\times\{1\}\}$, however, $p(s)\to \bar p\notin K^* +F^\perp$.

For the first relation, observe that $F^\perp = \lspan \{(0,0,1,0)\}$, and therefore
$$
r(s) := (0,0,\frac{4}{s},0)\in F^\perp.
$$
Hence, $p(s) = q(s) + r(s)$, where $r(s) \in F^\perp$, and we will next show that $q(s) \in K^*$.

We have explicitly
$$
q(s) = \left(2(\sqrt{7}+1)s, (5-\sqrt{7}),-4/s, (\sqrt{7}+3)s^2 \right).
$$

Abusing the notation and denoting by $\gamma_1$ the lifted version of the relevant curve, we have
\begin{align*}
\langle \gamma_1(u),q(s)\rangle
& = (\sqrt{7}+3 + 4 \frac{u}{s})(u-s)^2>0
\end{align*}
when $u\neq s$, also for $\gamma_2$ substituting $\varphi(u) =1/3 (2  + \sqrt{7} )u $,
\begin{align*}
\langle \gamma_2(\varphi(u)),q(s)\rangle
& = (3+\sqrt{7})(u-s)^2, \end{align*}
which is greater than zero unless $u=s$. We have hence shown that the point $q(s)$ is in the dual cone.

Let
$$
\bar p  = \lim_{s\downarrow 0} p(s) = (0,5-\sqrt{7},0,0),
$$
then
$$
\langle \bar p , \gamma_1(s) \rangle = (\sqrt{7}-5)s<0,
$$
and hence $\bar p\notin K^*$.

\end{proof}

\begin{proposition}[Properties of the cone $K$ from Example~\ref{eg:lost-tangent}]\label{prop:eg3} Let $K := \cone \{C\times \{1\}\}$, where $C := \conv\{\gamma_1,\gamma_2\}$, $\gamma_1(t) = (\cos t, \sin t, 1), \; t\in [0,\pi/2]$, $ \gamma_2(t) = (\cos t, \sin t, -1)\; t\in [0,\pi]$. The closed convex cone $K$ is
\begin{itemize}
\item facially exposed;
\item not strongly tangentially exposed;
\item FDC.
\end{itemize}
\end{proposition}
\begin{proof}

\textbf{To prove that the cone $K$ is facially exposed}, we use the same techniques as in the proof of Proposition~\ref{prop:eg2}.

The two-dimensional faces of $C$ are
$$
F_1 = \conv\{\gamma_1\}, \quad F_2 = \conv\{\gamma_2\}, \quad F_3 = \conv \{\gamma_1(0), \gamma_2(0), \gamma_2(\pi)\}, \quad F_4 = \conv\{\gamma_1(0), \gamma_1(\pi/2), \gamma_2(\pi)\};
$$
the one-dimensional faces are the line segments connecting $\gamma_1$ and $\gamma_2$,
$$
F_{11}(t) = \conv\{\gamma_1(t), \gamma_2(t)\}, \; t\in [0,\pi/2]; \quad F_{12}(t) = \conv\{\gamma_1(\pi/2), \gamma_2(t)\},\; t\in (\pi/2,\pi];
$$
and the remaining intersections of the two-dimensional faces,
$$
F_{13} = \conv\{\gamma_1(0), \gamma_1(\pi/2)\}, \quad F_{14} = \conv\{\gamma_2(0), \gamma_2(\pi)\}, \quad F_{15} = \conv\{\gamma_1(0), \gamma_2(\pi)\}.
$$

It is a technical exercise to verify that the two-dimensional faces $F_i$, $i\in \{1,\dots, 4\}$ are exposed by the hyperplanes that correspond to the following half-spaces that contain $C$,
$$
\langle  (0,0,1) \cdot \rangle\leq 1,\qquad 
\langle (0,0,-1), \cdot \rangle\leq 1,\qquad 
\langle (-1,-1,1), \cdot \rangle\leq 0,\qquad
\langle (0,-1,0), \cdot\rangle\leq 0.
$$
This also proves that the one-dimensional faces $F_{13}$, $F_{14}$, $F_{15}$ are exposed, by Proposition~\ref{prop:expint}.
The remaining families of one-dimensional faces $F_{11}$ and $F_{12}$ are exposed by the following two families of half-spaces and relevant hyperplanes,
$$
\langle (\cos t,\sin t,0), \cdot \rangle  \leq 1\,:\, t\in [0,\pi/2],
$$
$$
\langle (\cos \tau ,\sin \tau, \frac{1-\sin \tau}{2}), \cdot \rangle \leq \frac{1+\sin \tau}{2}, \tau\in (\pi/2,\pi].
$$

It is evident from using the same argument as in the proof of Proposition~\ref{prop:eg2} and invoking Proposition~\ref{prop:relint} together with the facial topology shown in Fig.~\ref{fig:sphere3},
\begin{figure}[ht]
{\centering
\includegraphics[scale=1]{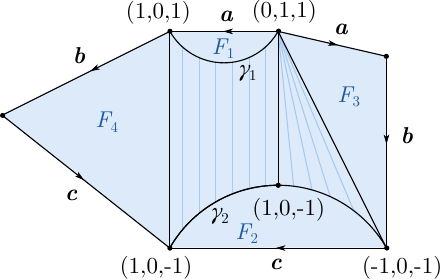}\\
}
\caption{Boundary of $C$  identified with the unit sphere}
\label{fig:sphere3}
\end{figure}
that the listed one- and two-dimensional faces together with their zero-dimensional intersections along the curves $\gamma_1$ and $\gamma_2$ comprise all nonempty proper faces of the set $C$. The exposure of the zero-dimensional faces follows from Proposition~\ref{prop:expint}.

\textbf{To prove that the cone $K$ is FDC} we begin with computing the polar cone explicitly. We can do this from the half-space description obtained earlier and using Propositions~\ref{prop:boundary} and~\ref{prop:polarlift}. The dual cone $K^*$ for $K$ is
\begin{align*}
K^\circ =  \cone \{& \{(-\cos t,-\sin t,0,1)\,:\, t\in [0,\pi/2]\},\\
& \left\{(-\cos \tau ,-\sin \tau, \frac{\sin \tau-1}{2},\frac{1+\sin \tau}{2}), \tau\in (\pi/2,\pi] \right\},\\
&(0,0,-1,1),
(0,0,1,1),
(1,1,-1,0),
(0,1,0,0)\}.
\end{align*}

To check whether $K$ is facially dual complete, it remains to consider all possible sums $F^\perp + K^*$ for orthogonal complements of faces of $K$ and see if these sets are closed.

Notice that whenever the face $F$ is one-dimensional, its orthogonal complement is a three-dimensional subspace. Its sum with any closed cone is closed, since the relevant one-dimensional projection of a closed cone is closed. By Proposition~\ref{prop:no-need-to-check}  all two-dimensional faces of $K$ also verify the closedness condition.

Due to our observation about one-dimensional faces and Proposition~\ref{prop:no-need-to-check} to prove that the cone $K = \cone \{C\times \{1\}\}$ is FDC we only need to check the closedness of $F^\perp+K^*$ for the three-dimensional faces of $K$ (that correspond to the two dimensional faces of $C$ shown in Fig~\ref{fig:twodfaces}).%
\begin{figure}[ht]
\centering
\includegraphics[scale=1]{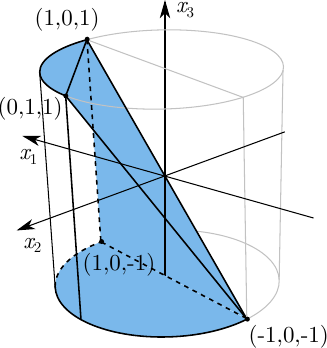}\\
\caption{Two dimensional faces of $C$}
\label{fig:twodfaces}
\end{figure}

For the three-dimensional faces of $K$ that correspond to the top and bottom faces $F_{11}$ and $F_{12}$ of the set $C$, we use Proposition~\ref{prop:projection} to reduce checking that the sum $F^\perp+ K^*$ is closed to checking that $\Pi_{\lspan F^\perp} K^*$ is closed.

To compute the projections we use a coordinate transformation that rotates the space so that $F^\perp$ coincides with $\lspan (0,0,0,1)$. This allows us to obtain a three-dimensional graphic representation of the projection for each case.

We use the representation $K^* = \cone S$, where
$$
S = S_1 \cup S_2 \cup S_3,
$$
\begin{align*}
S_1 & = \{(-\cos t,-\sin t,0,1)\,:\, t\in [0,\pi/2]\},\\
S_2 & = \left\{(-\cos \tau ,-\sin \tau, \frac{\sin \tau-1}{2},\frac{1+\sin \tau}{2}), \tau\in [\pi/2,\pi] \right\},\\
S_3 & = \{(0,0,-1,1),
(0,0,1,1),
(1,1,-1,0),
(0,1,0,0),
(0,0,0,1)\}.
\end{align*}

For the top face we have the corresponding face $F'_{11} = \cone \{F_{11}\times \{1\}\} = \cone \{\gamma_1\times \{1\}\}\lhd K$, and so
$$
\lspan {F'_{11}}  = \lspan \{(1,0,1,1),(0,1,1,1),(0,0,1,1)\}, \qquad {F'_{11}}^\perp = \lspan (0,0,1,-1).
$$

It is a technical exercise in linear algebra to verify that $U(F'_{11}) =  \cone S'$, where  $S' = \{S'_1,S'_2,S'_3\}$,
\begin{align*}
S_1'
 & = \left\{\{(-\cos t,-\sin t,1/\sqrt{2})\,:\, t\in [0,\pi/2]\}\right\},\\
S_2' & = \left\{(-\cos \tau ,-\sin \tau, 1/\sqrt{2}\sin \tau), \tau\in [\pi/2,\pi] \right\},\\
S_3'
 & = \left\{(0,0,\sqrt{2}),
  (1,1,-1/\sqrt{2}),
  (0,1,0),
  (0,0,1/\sqrt{2})\right\}.
\end{align*}

To show that $U(F'_{11})$ is closed, we use Proposition~\ref{prop:conichull}. It is easy to see that for $w = (1,1,z)$, where $z\in (2,2\sqrt{2})$, we have 
$$
\langle w,x\rangle >0 \quad \forall x\in S'.
$$

For the bottom face $F_{12}$ we have $F'_{12} = \cone \{\gamma_1\times \{1\}\}$, and the relevant linear subspaces are
$$
\lspan {F'_{12}}  = \lspan \{(1,0,1,-1),(0,1,1,-1),(0,0,1,-1)\}, \quad {F'_{12}}^\perp = \lspan (0,0,1,1).
$$
After computing the relevant unitary transformation $U$, the projection is a three dimensional set $U(F'_{12})= \cone S'$, where  $S' = \{S'_1,S'_2,S'_3\}$,
\begin{align*}
S_1'
 & = \left\{\{(\cos t,\sin t,1/\sqrt{2})\,:\, t\in [0,\pi/2]\}\right\},\\
S_2' & = \left\{(\cos \tau ,\sin \tau, 1/\sqrt{2}), \tau\in [\pi/2,\pi] \right\},\\
S_3'
 & = \left\{(0,0,\sqrt{2}),
  (-1,-1,1/\sqrt{2}),
  (0,-1,0),
  (0,0,1/\sqrt{2})\right\}.
\end{align*}

For $w = (0,y,-1)$, where $y\in (0,1/\sqrt{2})$, it is easy to check that $\langle w,x\rangle <0$ for all points in $S'$, and hence, by Proposition~\ref{prop:conichull} the set $\cone S'$ is closed. 

The remaining triangular faces satisfy Proposition~\ref{prop:Pataki}: since the triangular faces are polyhedral, their duals are also polyhedral, and have all their proper faces exposed.

\end{proof}

\end{document}